\documentclass[12pt]{article}
\PassOptionsToPackage{dvipsnames,svgnames,x11names}{xcolor}
\usepackage{xcolor}
\definecolor{labelkey}{rgb}{0,0.08,0.45}
\definecolor{refkey}{rgb}{0,0.6,0.0}
\definecolor{Brown}{rgb}{0.45,0.0,0.05}
\definecolor{lime}{rgb}{0.00,0.8,0.0}
\definecolor{lblue}{rgb}{0.5,0.5,0.99}
\definecolor{OliveGreen}{rgb}{0,0.6,0}
\definecolor{tyrianpurple}{rgb}{0.4, 0.01, 0.24}
\definecolor{myseagreen}{HTML}{3FBC9D}
\definecolor{myblue}{rgb}{0.9,0.9,0.98}
\colorlet{hlcyan}{cyan!30}

\usepackage{mathpazo} 
\usepackage{subcaption} 
\usepackage[T1]{fontenc}
\usepackage{mathtools}
\usepackage{amssymb}
\usepackage{stmaryrd}
\usepackage{empheq}

\usepackage{hyperref}
\hypersetup{
  colorlinks=true,
  linkcolor=refkey,
  urlcolor=lblue,
  citecolor=magenta
}

\usepackage{amsthm}
\makeatletter
\def\th@plain{%
  \thm@notefont{}%
  \itshape 
}
\def\th@definition{%
  \thm@notefont{}%
  \normalfont 
}

\usepackage[capitalize,nameinlink]{cleveref}
\makeatother
\newtheorem{theorem}{Theorem}[section]
\newtheorem{lemma}[theorem]{Lemma}
\newtheorem{corollary}[theorem]{Corollary}
\newtheorem{proposition}[theorem]{Proposition}
\newtheorem{definition}[theorem]{Definition}
\newtheorem{example}[theorem]{Example}
\newtheorem{fact}[theorem]{Fact}
\newtheorem{remark}[theorem]{Remark}
\crefname{theorem}{Theorem}{Theorems}
\Crefname{theorem}{Theorem}{Theorems}
\crefname{fact}{Fact}{facts}
\Crefname{fact}{Fact}{facts}
\crefname{equation}{}{equations}
\crefname{chapter}{Appendix}{chapters}
\crefname{item}{}{items}
\crefname{enumi}{}{}

\def\endproof{\ensuremath{\hfill \quad \blacksquare}}

\usepackage{nicematrix}
\usepackage{booktabs}
\usepackage{siunitx}
\usepackage{tikz}
\usepackage[color=myseagreen]{todonotes}
\usepackage{soul}
\usepackage{nameref}
\usepackage{graphicx}
\usepackage{float}
\usepackage[shortlabels,inline]{enumitem}
\setlist[enumerate]{nosep}
\usepackage{relsize}

\usepackage[margin=1in,footskip=0.4in]{geometry}
\makeatletter

\let\orig@label\label

\renewcommand{\label}[1]{%
  \begingroup
  \def\@currentlabelname{}%
  \ifx\current@theorem\relax\else
    \def\@currentlabelname{\current@theorem}%
  \fi
  \ifx\cref@currentlabel\undefined\else
    \let\@currentlabelname\cref@currentlabel
  \fi
  \orig@label{#1}%
  \endgroup
}

\AtBeginEnvironment{theorem}{\def\current@theorem{theorem}}
\AtBeginEnvironment{proposition}{\def\current@theorem{proposition}}
\AtBeginEnvironment{fact}{\def\current@theorem{fact}}
\makeatother

\newcommand{\seppthree}{\setlength{\itemsep}{-3pt}}

\newcommand{\nnn}{\ensuremath{{n\in\mathbb{N}}}}

\newcommand{\menge}[2]{\big\{{#1}~\big|~{#2}\big\}}

\newcommand{\scal}[2]{\left\langle{#1},{#2}\right\rangle}

\newcommand{\RR}{\ensuremath{\mathbb{R}}}

\newcommand{\NN}{\ensuremath{\mathbb{N}}}

\newcommand{\conv}{\ensuremath{\operatorname{conv}\,}}

\newcommand{\Fix}{\ensuremath{\operatorname{Fix}}}
\newcommand{\Id}{\ensuremath{\operatorname{Id}}}

\providecommand{\norm}[1]{\lVert#1\rVert}

\providecommand{\innp}[1]{\langle#1\rangle}

\newcommand{\mybluebox}[1]{\colorbox{myblue}{\hspace{1em}#1\hspace{1em}}}

\author{
Heinz H.\ Bauschke\thanks{
Mathematics, University
of British Columbia,
Kelowna, B.C.\ V1V~1V7, Canada. E-mail:
\texttt{heinz.bauschke@ubc.ca}.}~~~and~
Tran Thanh Tung\thanks{
Mathematics, University
of British Columbia,
Kelowna, B.C.\ V1V~1V7, Canada. E-mail:
 E-mail: \texttt{tung.tran@ubc.ca}.}
}

\title{\textsf{
    Extending Meshulam's result on the boundedness of 
    orbits of relaxed projections onto affine subspaces\\  
    from finite to infinite-dimensional Hilbert spaces}
}

\date{June 2, 2026}

\begin{document}
\allowdisplaybreaks
\maketitle
\begin{abstract}
    \noindent
    In 1996, Meshulam proved that any sequence generated in Euclidean space by randomly projecting onto affine subspaces drawn from a finite collection stays bounded even
    if the intersection of the subspaces is empty. His proof,
    which works even for relaxed projections, relies on an ingenious induction on the dimension of the Euclidean space.

    In this paper, we extend Meshulam's result to the general Hilbert space setting by an induction proof on the number
    of affine subspaces in the given collection. We require
    that the corresponding parallel linear subspaces
    are innately regular --- this assumption always holds in Euclidean space. We also discuss the sharpness of our result and make a connection to randomized block Kaczmarz methods.
\end{abstract}



\section{Introduction}
Throughout this paper,
\begin{empheq}[box=\mybluebox]{equation}
    \label{e:mcA}
    \text{$X$ is a real Hilbert space, with inner product $\scal{\cdot}{\cdot}$ and induced norm $\norm{\cdot}$,}
\end{empheq}
and
\begin{empheq}[box=\mybluebox]{equation}
    \label{e:mcL}
    \text{$\mathcal{A}$ is a nonempty finite collection of closed affine\footnotemark subspaces of $X$}
\end{empheq}
with
\begin{empheq}[box=\mybluebox]{equation}
    \text{$\mathcal{L}$ is the collection of closed linear subspaces associated with $\mathcal{A}$.}
\end{empheq}
\footnotetext{Recall that a subset $A$ of $X$ is affine if $A-A$ is a linear subspace.}
Given a nonempty finite collection
$\mathcal{C}$ of closed convex subsets of $X$ and an interval $\Lambda \subseteq [0,2]$, consider the associated set of relaxed projectors\footnotemark
\begin{empheq}[box=\mybluebox]{equation}
    \label{RCLambda}
    \mathcal{R}_{\mathcal{C},\Lambda} := \menge{(1-\lambda)\Id+\lambda P_C}{C\in\mathcal{C},\ \lambda\in \Lambda},
\end{empheq}
where $P_C$ is the orthogonal projector onto $C$ and $\Id$ is the identity mapping on $X$. {For notational convenience, we will write $\mathcal{R}_{\mathcal{C},\lambda}$ when $\Lambda =[\lambda,\lambda] = \{\lambda\}$.}
\footnotetext{Given a nonempty closed convex subset $C$ of $X$, we denote by $P_C$ the operator which maps $x\in X$ to its unique nearest point in $C$. The relaxed projector is given by $(1-\lambda)\Id+\lambda P_C$. When $\lambda=1$, this reduces to the usual projection mapping; as $\lambda\to 2^-$, it approaches a reflection; and as $\lambda\to 0^+$, it approaches the identity mapping.}

Building on the work of Aharoni, Duchet, and Wajnryb \cite{ADW},
Meshulam proved\footnote{More precisely,
    Meshulam provided a complete proof for $\Lambda = \{1\}$
    and he pointed out that Dr.\ Ron Aharoni noted that the proof extends
    to the more general setting.} in \cite[Theorem~2]{Meshulam} the following result:

\begin{fact}[Meshulam]
    \label{f:Meshulam}
    Suppose that $X$ is finite-dimensional.
    Let $\lambda\in[0,2[$ and $x_0\in X$.
    Generate the sequence $(x_n)_{\nnn}$ in $X$ as follows:
    Given a current term $x_n$, pick $R_n \in \mathcal{R}_{\mathcal{A},[0,\lambda]}$, and update via
    \begin{equation}
        x_{n+1} := R_n x_n.
    \end{equation}
    Then the sequence $(x_n)_{\nnn}$ is bounded.
\end{fact}

This result is easy to prove if
$\bigcap_{A\in \mathcal{A}} A \neq \varnothing$,
because each $P_A$ is (firmly) nonexpansive\footnote{Recall that an operator $T:X\to X$ is called firmly nonexpansive if
    $(\forall x\in X)(\forall y\in X)$ $\|Tx-Ty\|^2\leq \|x-y\|^2-\|(\Id-T)x-(\Id-T)y\|^2$; equivalently, $\|Tx-Ty\|^2\leq\langle x-y,Tx-Ty\rangle$.}, so
the sequence $(x_n)_\nnn$ is Fej\'{e}r monotone\footnote{Given a nonempty subset $C$ of $X$, recall that a sequence $(x_n)_{n\in\mathbb{N}}$ is called Fej\'er monotone with respect to $C$ if $(\forall c\in C)(\forall n\in\NN)$ $\|x_{n+1}-c\|\leq \|x_n-c\|$.} with
respect to this intersection (and hence, bounded). In fact, convergence in this case was also established in \cite[Theorem~3.3]{HBrandom}.

Meshulam's proof of \cref{f:Meshulam} in the case when
$\bigcap_{A\in \mathcal{A}} A = \varnothing$ is
much more involved and relies on a clever
\emph{induction on the dimension} of the space~$X$.
His proof thus does not generalize to the case when
$X$ is \emph{infinite-dimensional}, which motivates the
following:

\emph{The goal of this paper is to generalize Meshulam's result to the case when $X$ is \emph{infinite-dimensional}.}

More precisely,
in \cref{t:meshulaminfvar},
we extend \cref{f:Meshulam} to the case where $X$ is potentially infinite-dimensional under the additional assumption of
``innate regularity'' of the collection $\mathcal{L}$.
This assumption is automatically true when $X$ is finite-dimensional; moreover, it is known that \emph{some}
additional assumption is required in general
(see {\cref{e:noregular}}).
Similar to Meshulam's proof, we argue by mathematical induction. In stark contrast, Meshulam's induction is on
the \emph{dimension} of $X$ while our proof features
an induction is on  \emph{the number of
    closed affine subspaces} in $\mathcal{A}$.

The rest of the paper is organized as follows. After discussing some auxiliary results in
\cref{s:aux}, we present the key ingredients for the proof of our main result in \cref{s:relaxedproj}. The extension of \cref{f:Meshulam} to infinite-dimensional spaces
is presented in
\cref{s:fixedlambda} (see \cref{t:meshulaminfvar}). In the final \cref{s:ex}, we comment on a nice connection
to randomized block Kaczmarz methods, on an example in $L_2[0,1]$,  and on limiting examples.
Moreover, we present a linear convergence result for a fixed composition of relaxed projectors as well as an
illustration comparing this to \cref{t:meshulaminf}.

The notation we employ is standard and follows, e.g., \cite{Bauschke2017}
and \cite{Rocky}.

\section{Auxiliary results}
\label{s:aux}
From this section onward, for a closed convex subset $C$ of $X$ and a constant $\lambda$, we will use the following notation:
\begin{empheq}[box=\mybluebox]{equation}
    R_{C,\lambda} := (1-\lambda)\Id + \lambda P_C.
\end{empheq}

\begin{fact}
    {\rm \cite[Corollary~3.24]{Bauschke2017}}
    \label{f:decomposition} Let $L$ be a
    closed linear subspace of $X$. Then,
    \begin{equation}
        \Id=P_L+P_{L^\perp},
    \end{equation}
    and
    \begin{equation}
        (\forall x\in X)\quad \norm{x}^2=\norm{P_{L}x}^2+\norm{P_{L^\perp}x}^2=\norm{P_{L}x}^2+\norm{x-P_Lx}^2.
    \end{equation}
    In particular, we have $\norm{P_Lx}\leq\norm{x}$ for all $x\in X$.
\end{fact}

\begin{fact}\label{f:reflection1}
    Let $L$ be a closed linear subspace of $X$, and let $\lambda\in\left]0,2\right[$. Then,
    \begin{equation}
        (\forall x\in X)\quad\frac{\norm{x}^2-\norm{R_{L,\lambda}x}^2}{\lambda(2-\lambda)}=\norm{x-P_Lx}^2.
    \end{equation}
    Consequently, $\norm{R_{L,\lambda}x}\leq\norm{x}$ for all $x\in X$.
\end{fact}
\begin{proof}
    By definition, we have
    \begin{subequations}
        \begin{align}
            \norm{R_{L,\lambda}x}^2 & =\norm{x-\lambda(x-P_Lx)}^2                                                                  \\
                                    & =\norm{x}^2-2\lambda\innp{x,x-P_Lx}+\lambda^2\norm{x-P_Lx}^2                                 \\
                                    & =\norm{x}^2-2\lambda\left(\norm{x-P_Lx}^2+\innp{P_Lx,x-P_Lx}\right)+\lambda^2\norm{x-P_Lx}^2 \\
                                    & =\norm{x}^2-2\lambda\norm{x-P_Lx}^2+\lambda^2\norm{x-P_Lx}^2                                 \\
                                    & =\norm{x}^2-\lambda(2-\lambda)\norm{x-P_Lx}^2,
        \end{align}
    \end{subequations}
    which is the desired result. The ``Consequently'' part then follows from $\lambda\in\left]0,2\right[$.
\end{proof}

Let $x,y\in X$. We adopt the convention that the angle between $x$ and $y$ is $\frac{\pi}{2}$ if exactly one of them is the zero vector, and $0$ if $x=y=0$.

\begin{proposition}\label{p:sine1}
    Let $L$ be a closed linear subspace of $X$, and let $\lambda\in\left]0,2\right[$. Then for all $x\in X$, the sine {$\sin_L(x)$ and cosine $\cos_L(x)$} of the angle between $x$ and its projection $P_Lx$ are given by
    \begin{equation}
        \sin_L(0)=0,\quad\cos_L(0)={1,}
    \end{equation}
    and
    \begin{equation}
        (\forall x\in X\smallsetminus\{0\})\quad\sin_L(x)=\frac{\norm{x-P_Lx}}{\norm{x}},\quad \cos_L(x)=\frac{\norm{P_Lx}}{\norm{x}}.
    \end{equation}
    Moreover, we have
    \begin{equation}
        (\forall x\in X\smallsetminus\{0\})(\forall \varepsilon\in\left[0,1\right])\quad\sin_L(x)\geq\varepsilon\Leftrightarrow\norm{R_{L,\lambda}x}\leq\sqrt{1-\lambda(2-\lambda)\varepsilon^2}\norm{x}.
    \end{equation}
\end{proposition}
\begin{proof}
    The case $x=0$ is clear. When $x\neq0$ and $P_Lx=0$, then, by the angle convention, we get that
    \begin{equation}
        \sin_L(x)=1=\frac{\norm{x-P_Lx}}{\norm{x}}\quad\text{and}\quad\cos_L(x)=0=\frac{\norm{P_Lx}}{\norm{x}}.
    \end{equation}

    When $x\neq0$ and $P_Lx\neq0$, we have that the cosine of the angle between $x$ and $P_Lx$ is given by
    \begin{align}
        \cos_L(x) & =\frac{\innp{x,P_Lx}}{\norm{x}\norm{P_Lx}}
        =\frac{\norm{P_Lx}^2}{\norm{x}\norm{P_Lx}}
        =\frac{\norm{P_Lx}}{\norm{x}}.\label{20251111a}
    \end{align}
    Since the sine of the angle between any two vectors is always nonnegative,
    we obtain
    \begin{equation}
        \sin_L(x)=\sqrt{1-\cos_L^2(x)}\overset{\cref{20251111a}}{=}\sqrt{1-\frac{\norm{P_Lx}^2}{\norm{x}^2}}=\frac{\sqrt{\norm{x}^2-\norm{P_Lx}^2}}{\norm{x}}\overset{\text{\cref{f:decomposition}}}{=}\frac{\norm{x-P_Lx}}{\norm{x}}.\label{20251111b}
    \end{equation}
    This implies
    \begin{subequations}
        \begin{align}
            (\forall\varepsilon\in\left[0,1\right])\quad\sin_L(x)\geq\varepsilon & \Leftrightarrow\sin_L^2(x)\geq\varepsilon^2\notag                                                                                    \\
                                                                                 & \Leftrightarrow\frac{\norm{x-P_Lx}^2}{\norm{x}^2}\geq\varepsilon^2\tag{by \cref{20251111b}}                                          \\
                                                                                 & \Leftrightarrow\frac{\norm{x}^2-\norm{R_{L,\lambda}x}^2}{\lambda(2-\lambda)\norm{x}^2}\geq\varepsilon^2\tag{by \cref{f:reflection1}} \\
                                                                                 & \Leftrightarrow\norm{R_{L,\lambda}x}\leq\sqrt{1-\lambda(2-\lambda)\varepsilon^2}\norm{x}\notag,
        \end{align}
    \end{subequations}
    which is the desired result.
\end{proof}

\begin{fact}
    \label{f:commute}
    Let $L_1$ and $L_2$ be two closed linear subspaces of $X$ such that $L_1\subseteq L_2$, and let $\lambda\in\left]0,2\right[$. Then
    \begin{equation}
        P_{L_1}=P_{L_1}P_{L_2}=P_{L_2}P_{L_1},\quad P_{L_1^\perp}P_{L_2}=P_{L_2}P_{L_1^\perp},\quad P_{L_1}P_{L_2^\perp}=P_{L_2^\perp}P_{L_1}.\label{20251121d}
    \end{equation}
    Consequently,
    \begin{equation}
        P_{L_1}=P_{L_1}R_{L_2,\lambda}=R_{L_2,\lambda}P_{L_1},\quad P_{L_1^\perp}R_{L_2,\lambda}=R_{L_2,\lambda}P_{L_1^\perp},\quad P_{L_1}R_{L_2^\perp,\lambda}=R_{L_2^\perp,\lambda}P_{L_1}.
    \end{equation}
\end{fact}
\begin{proof}
    \cite[Lemma~9.2]{Deutsch} yields \cref{20251121d}.
    The ``Consequently'' part then follows.
\end{proof}

\begin{proposition}
    {\rm \cite[Lemma~3.1]{GT}}
    \label{p:sine2}
    Let $L_1$ and $L_2$ be two closed linear subspaces of $X$ such that $L_1\subseteq L_2$, and let $\lambda\in\left]0,2\right[$. Then,
    \begin{equation}
        (\forall x\in X)\quad\sin_{L_1}(R_{L_2,\lambda}x)\leq\sin_{L_1}(x).
    \end{equation}
\end{proposition}
\begin{proof}
    Let $x\in X$. The case in which $x=0$ or $R_{L_2,\lambda}x=0$ is clear.

    When $x\neq0$ and $R_{L_2,\lambda}x\neq0$, by \cref{p:sine1}, we have that
    \begin{equation}
        \cos_{L_1}(x)=\frac{\norm{P_{L_1}x}}{\norm{x}}\quad\text{and}\quad\cos_{L_1}(R_{L_2,\lambda}x)=\frac{\norm{P_{L_1}R_{L_2,\lambda}x}}{\norm{R_{L_2,\lambda}x}}.
    \end{equation}
    By \cref{f:commute} and \cref{f:reflection1}, we obtain
    \begin{equation}
        \cos_{L_1}(R_{L_2,\lambda}x)=\frac{\norm{P_{L_1}x}}{\norm{R_{L_2,\lambda}x}}\geq\frac{\norm{P_{L_1}x}}{\norm{x}}=\cos_{L_1}(x).
    \end{equation}
    This yields $\sin_{L_1}(R_{L_2,\lambda}x)\leq\sin_{L_1}(x)$.
\end{proof}

\begin{definition}[{regularity}]\label{d:regular}
    The collection $\mathcal{L}$ is said to be \emph{regular} if there exists a constant $\kappa>0$ such that\footnotemark
    \begin{equation}
        (\forall x\in X)\quad d_{\cap_{L\in\mathcal{L}}L}(x)\leq\kappa\max_{L\in\mathcal{L}}d_{L}(x).
    \end{equation}
\end{definition}

\begin{remark}\label{r:1}
    Note that this is equivalent to $\sum_{L\in\mathcal{L}}L^\perp$ being closed (see \cite[Theorem~5.19]{BB1996}), which automatically holds when $X$ is finite-dimensional.
\end{remark}

\begin{proposition}
    \label{p:sineregular}
    {\rm \cite[Corollary~3.2]{GT}}
    Let $\mathcal{L}$ be regular. Then, there exists a constant $\kappa>0$ such that
    \begin{equation}
        (\forall x\in X)\quad\sin_{\cap_{L\in\mathcal{L}}}(x)\leq\kappa\max_{L\in\mathcal{L}}\sin_{L}(x).
    \end{equation}
\end{proposition}
\begin{proof}
    Let $x\in X$. The case when $x=0$ is clear.

    When $x\neq0$, using \cref{p:sine1}, we obtain
    \begin{equation}
        \sin_{\cap_{L\in\mathcal{L}}}(x)=\frac{d_{\cap_{L\in\mathcal{L}}}(x)}{\norm{x}}.
    \end{equation}
    Then, by \cref{d:regular}, there exists a constant $\kappa>0$ such that
    \begin{equation}
        \sin_{\cap_{L\in\mathcal{L}}}(x)\leq\kappa\frac{\max_{L\in\mathcal{L}}d_L(x)}{\norm{x}}.
    \end{equation}
    This combined with \cref{p:sine1} yields the desired result.
\end{proof}

\begin{definition}[{innate regularity}]
    The collection $\mathcal{L}$ is said to be \emph{innately regular} if every subcollection of $\mathcal{L}$ is regular.
\end{definition}

\begin{remark}\label{r:2}
    Note that this is equivalent to $\sum_{L\in\widetilde{\mathcal{L}}} L^\perp$ being closed for all subcollections $\widetilde{\mathcal{L}}$ of $\mathcal{L}$ (see \cite{HBrandom} and especially \cite[Section~2]{GT} for a nice summary). Again, this condition automatically holds when $X$ is finite-dimensional. Also note that regularity and innate
    regularity do not coincide in general (see \cite{RZBams} and
    also \cite{RZJAA} for further information).
\end{remark}

\footnotetext{Given a subset $C$ of $X$, its distance function is defined by $d_C(x):=\inf\|x-C\|=\inf_{c\in C}\|x-c\|$.}
\section{Random product of relaxed projectors}\label{s:relaxedproj}

In this section, we develop several results which will
make the proof of the main result in the next section
more structured.

Let $\lambda\in\left]0,2\right[$, and let $x^{(0)}:=x\in X$. Consider the random relaxed projection sequence
\begin{equation}
    x^{(n+1)}:=R_n\cdots R_0x,\label{20251111c}
\end{equation}
where $R_n\in\mathcal{R}_{\mathcal{L},\lambda}$. Let $L_n\in\mathcal{L}$ be the subspace associated with $R_n$. For $q\in\mathbb{N}$, we define
\begin{empheq}[box=\mybluebox]{equation}
    \mathbf{L}_q:=\bigcap_{i=0}^q L_i \quad\text{and}\quad N_q:=\bigl|\{L_i\mid i\in\{0,\dots,q\}\}\bigr|.
\end{empheq}

\begin{proposition}
    {\rm \cite[Lemma~3.3 and Proposition~3.6]{GT}}
    \label{p:sine3}
    Suppose that $\mathcal{L}$ is innately regular. Then there exists a $\kappa_*>1$ such that
    \begin{equation}
        (\forall q\in\mathbb{N})(\forall x\in X)\quad\sin_{\mathbf{L}_{q}}\big(x^{(q)}\big)\leq\kappa_*^{N_{q}-1}\max_{i\in\left\{0,\dots,q\right\}}\sin_{L_i}\big(x^{(i)}\big).\label{20251111e}
    \end{equation}
    Moreover, for each $q\in\mathbb{N}$, we have\footnote{
        Here $R_{i-1}\cdots R_0=\Id$ when $i=0$, by the empty product convention.}
    \begin{equation}
        (\forall x\in X)(\exists i\in\left\{0,\dots,q\right\})\quad \norm{R_i\cdots R_0P_{\mathbf{L}_q^\perp}x}\leq\sqrt{1-\lambda(2-\lambda)\kappa_*^{-2(N_{q}-1)}}\norm{R_{i-1}\cdots R_0P_{\mathbf{L}_q^\perp}x};\label{20251111f}
    \end{equation}
    consequently, $\norm{R_q\cdots R_0P_{\mathbf{L}_q^\perp}}\leq\sqrt{1-\lambda(2-\lambda)\kappa_*^{-2(N_{q}-1)}}<1$.
\end{proposition}
\begin{proof}
    Let $\mathcal{L}_1$ and $\mathcal{L}_2$ be subcollections of $\mathcal{L}$. Note that by the innate regularity of $\mathcal{L}$ and \cref{r:2}, we have
    \begin{equation}
        \Big(\bigcap_{L\in\mathcal{L}_1}L\Big)^\perp+\Big(\bigcap_{L\in\mathcal{L}_2}L\Big)^\perp=\overline{\sum_{L\in\mathcal{L}_1}L^\perp}+\overline{\sum_{L\in\mathcal{L}_2}L^\perp}=\sum_{L\in\mathcal{L}_1}L^\perp+\sum_{L\in\mathcal{L}_2}L^\perp=\sum_{L\in\mathcal{L}_1\cup\mathcal{L}_2}L^\perp,\label{263003a}
    \end{equation}
    and $\sum_{L\in\mathcal{L}_1\cup\mathcal{L}_2}L^\perp$ is closed. This combined with \cref{263003a} and \cref{r:1} yields $\left\{\bigcap_{L\in\mathcal{L}_1}L,\bigcap_{L\in\mathcal{L}_2}L\right\}$ is regular.

    Let $\kappa_*$ be the maximum constant arising from \cref{p:sineregular} when applied to the collection $\left\{\bigcap_{L\in\mathcal{L}_1}L,\bigcap_{L\in\mathcal{L}_2}L\right\}$, where the maximum is taken over all subcollections $\mathcal{L}_1$ and $\mathcal{L}_2$ of $\mathcal{L}$. WLOG, we can assume $\kappa_*>1$.

    We will prove \cref{20251111e} by induction (on $q$).
    The base case ($q=0$) states that
    \begin{equation}
        (\forall x\in X)\quad\sin_{\mathbf{L}_{0}}\left(x\right)\leq\kappa_*^{N_{0}-1}\sin_{L_0}\left(x\right),
    \end{equation}
    and this is always true because $\mathbf{L}_0=L_0$ and $N_0=1$.

    Let $n\in\mathbb{N}$. Assume that the following statement holds true
    \begin{equation}
        (\forall x\in X)\quad\sin_{\mathbf{L}_{n}}\big(x^{(n)}\big)\leq\kappa_*^{N_{n}-1}\max_{i\in\left\{0,\dots,n\right\}}\sin_{L_i}\big(x^{(i)}\big).\label{20251111g}
    \end{equation}
    If $N_{n+1}=N_{n}$, then $\mathbf{L}_{n+1}=\mathbf{L}_{n}$.
    Hence, we obtain
    \begin{subequations}
        \begin{align}
            (\forall x\in X)\quad\sin_{\mathbf{L}_{n+1}}\big(x^{(n+1)}\big)=\quad\quad\  & \sin_{\mathbf{L}_{n}}\big(R_{n}x^{(n)}\big)                                           \\
            \overset{\text{\cref{p:sine2}}}{\leq}\                                       & \sin_{\mathbf{L}_{n}}\big(x^{(n)}\big)                                                \\
            \overset{\cref{20251111g}}{\leq}\quad\quad                                   & \kappa_*^{N_{n}-1}\max_{i\in\left\{0,\dots,n\right\}}\sin_{L_i}\big(x^{(i)}\big)      \\
            \leq\quad\;\;\;\;                                                            & \kappa_*^{N_{n+1}-1}\max_{i\in\left\{0,\dots,n+1\right\}}\sin_{L_i}\big(x^{(i)}\big).
        \end{align}
    \end{subequations}

    If $N_{n+1}=N_{n}+1$, then applying \cref{p:sineregular} to the collection $\left\{\mathbf{L}_{n},L_{n+1}\right\}$ yields, for all $x\in X$,
    \begin{subequations}
        \begin{align}
            \sin_{\mathbf{L}_{n+1}}\big(x^{(n+1)}\big)\leq\quad\quad\  & \kappa_*\max\Big\{\sin_{\mathbf{L}_{n}}\big(x^{(n+1)}\big),\sin_{L_{n+1}}\big(x^{(n+1)}\big)\Big\}                                                           \\
            \overset{\text{\cref{p:sine2}}}{\leq}\                     & \kappa_*\max\Big\{\sin_{\mathbf{L}_{n}}\big(x^{(n)}\big),\sin_{L_{n+1}}\big(x^{(n+1)}\big)\Big\}                                                             \\
            \overset{\cref{20251111g}}{\leq}\quad\quad                 & \kappa_*\max\Big\{\kappa_*^{N_{n}-1}\max_{i\in\left\{0,\dots,n\right\}}\sin_{L_i}\big(x^{(i)}\big),\sin_{L_{n+1}}\big(x^{(n+1)}\big)\Big\}                   \\
            \overset{\kappa_*>1}{\leq}\quad\:\:\:\,                    & \kappa_*\max\Big\{\kappa_*^{N_{n}-1}\max_{i\in\left\{0,\dots,n\right\}}\sin_{L_i}\big(x^{(i)}\big),\kappa_*^{N_{n}-1}\sin_{L_{n+1}}\big(x^{(n+1)}\big)\Big\} \\
            =\quad\;\;\;\;                                             & \kappa_*^{N_{n+1}-1}\max_{i\in\left\{0,\dots,n+1\right\}}\sin_{L_i}\big(x^{(i)}\big).
        \end{align}
    \end{subequations}
    Hence, \cref{20251111e} is proven.

    Since \cref{20251111e} is true for every sequence starting from $X$, it also holds for every sequence starting from $\mathbf{L}_q^\perp$, that is,
    \begin{equation}
        (\forall q\in\mathbb{N})(\forall x\in \mathbf{L}_q^\perp)\quad\sin_{\mathbf{L}_{q}}\big(R_q\cdots R_0x\big)\leq\kappa_*^{N_{q}-1}\max_{i\in\left\{0,\dots,q\right\}}\sin_{L_i}\big(R_{i-1}\cdots R_0x\big).
    \end{equation}
    This is equivalent to
    \begin{equation}
        (\forall q\in\mathbb{N})(\forall x\in X)\quad\sin_{\mathbf{L}_{q}}\big(R_q\cdots R_0P_{\mathbf{L}_q^\perp}x\big)\leq\kappa_*^{N_{q}-1}\max_{i\in\left\{0,\dots,q\right\}}\sin_{L_i}\big(R_{i-1}\cdots R_0P_{\mathbf{L}_q^\perp}x\big).\label{20251120a}
    \end{equation}

    By \cref{f:commute}, we obtain
    \begin{equation}
        R_q\cdots R_0P_{\mathbf{L}_q^\perp}x=P_{\mathbf{L}_q^\perp}R_q\cdots R_0x\in\mathbf{L}_q^\perp.\label{20251120b}
    \end{equation}
    If $R_q\cdots R_0P_{\mathbf{L}_q^\perp}x=0$, then
    \begin{equation}
        \norm{R_q\cdots R_0P_{\mathbf{L}_q^\perp}x}=0\leq\sqrt{1-\lambda(2-\lambda)\kappa_*^{-2(N_{q}-1)}}\norm{R_{q-1}\cdots R_0P_{\mathbf{L}_q^\perp}x}.
    \end{equation}
    For $R_q\cdots R_0P_{\mathbf{L}_q^\perp}x\neq0$, \cref{20251120b} and \cref{p:sine1} imply $\sin_{\mathbf{L}_q}\big(R_q\cdots R_0P_{\mathbf{L}_q^\perp}x\big)=1$. It follows from \cref{20251120a} that for each $q\in\mathbb{N}$:
    \begin{equation}
        (\exists i\in\{0,\dots,q\})\quad\sin_{L_i}\big(R_{i-1}\cdots R_0P_{\mathbf{L}_q^\perp}x\big)\geq\kappa_*^{-(N_q-1)}.
    \end{equation}
    By \cref{p:sine1}, this is equivalent to
    \begin{equation}
        (\exists i\in\{0,\dots,q\})\quad\norm{R_i\cdots R_0P_{\mathbf{L}_q^\perp}x}\leq\sqrt{1-\lambda(2-\lambda)\kappa_*^{-2(N_{q}-1)}}\norm{R_{i-1}\cdots R_0P_{\mathbf{L}_q^\perp}x}.
    \end{equation}
    Since relaxed projectors are linear and nonexpansive (see \cref{f:reflection1}), the ``consequently'' part then follows.
\end{proof}

\begin{remark}
    For $\lambda = 1$
    and $q\in\mathbb{N}\smallsetminus\{0\}$,
    the conclusion of
    \cref{p:sine3} holds with
    $i \in \{1,\dots,q\}$.
\end{remark}
\begin{proof}
    Let $\kappa_*$ be the maximum constant arising from \cref{p:sineregular} when applied to the collection $\left\{\bigcap_{L\in\mathcal{L}_1}L,\bigcap_{L\in\mathcal{L}_2}L\right\}$, where the maximum is taken over all subcollections $\mathcal{L}_1$ and $\mathcal{L}_2$ of $\mathcal{L}$.
    WLOG, we assume that $\kappa_*>1$.

    We will prove this by induction (on $q$). The base case $q=1$ now states:
    \begin{equation}
        (\forall x\in X)\quad\sin_{\mathbf{L}_{1}}\big(x^{(1)}\big)\leq\kappa_*^{N_{1}-1}\max\Big\{\sin_{L_1}\big(x^{(1)}\big)\Big\}.
    \end{equation}
    If $N_1=N_0=1$, i.e., $L_1=L_0$, then $\sin_{\mathbf{L_1}}\big(x^{(1)}\big)=\sin_{{L_1}}\big(x^{(1)}\big)$ and we are done.

    If $N_1=N_0+1=2$, then
    \begin{subequations}
        \begin{align}
            (\forall x\in X)\quad\sin_{\mathbf{L}_{1}}\big(x^{(1)}\big) & =\sin_{L_0\cap L_1}\big(R_{0}x\big)\notag                                                                       \\
                                                                        & \leq\kappa_*\max\Big\{\sin_{L_0}\big(R_{0}x\big),\sin_{L_1}\big(x^{(1)}\big)\Big\}\tag{by \cref{p:sineregular}} \\
                                                                        & =\kappa_*\max\Big\{\sin_{L_0}\big(P_{L_0}x\big),\sin_{L_1}\big(x^{(1)}\big)\Big\}\tag{since $\lambda=1$}        \\
                                                                        & =\kappa_*\max\Big\{0,\sin_{L_1}\big(x^{(1)}\big)\Big\}\tag{by \cref{p:sine1}},
        \end{align}
    \end{subequations}
    which completes the proof of the base case. The remaining part of the proof is identical to that of \cref{p:sine3}, except that here $i \in \{1, \dots, q\}$.
\end{proof}

\begin{remark}
    In the original paper, \cref{p:sine3} is stated with
    $\sin_{\mathbf{L}_{q}}(x^{(q+1)})$ and $\kappa_*^{N_q}$ instead of
    $\sin_{\mathbf{L}_{q}}(x^{(q)})$ and $\kappa_*^{N_q-1}$; however, the proof is essentially the same.
\end{remark}

{
    We now introduce a notion that will be useful not
    only in reformulating \cref{p:sine3} but also
    in the proof of \cref{t:meshulaminf}:
}

\begin{definition}[cycle]
    \label{def:cycle}
    Let $\lambda\in\left]0,2\right[$. A finite product $Q$ of relaxed projectors in $\mathcal{R}_{\mathcal{L},\lambda}$ that satisfies both
    \begin{subequations}
        \begin{align}
             & \text{for every $L\in\mathcal{L}$, the
                   relaxed projector $R_{L,\lambda}$
                   appears at least once in $Q$, and }
            \label{eq:cycle1}                                  \\
             & \text{there exists $L\in\mathcal{L}$ such
                   that the relaxed projector $R_{L,\lambda}$
                   appears exactly once in $Q$},
            \label{eq:cycle2}
        \end{align}
    \end{subequations}
    is called a \emph{cycle}. We denote by $\mathcal{Q}$ the set of all cycles\footnotemark.
    \footnotetext{{Technically speaking, $\mathcal{Q}$
            depends on $\lambda$; however, in our usage,
            the underlying $\lambda$ will be clear from the context.}}
\end{definition}

\begin{example}
    Suppose that $\mathcal{L}=\{L_1,L_2,L_3\}$ and $\lambda\in\left]0,2\right[$. The following products are all cycles:
    \begin{equation}
        R_{L_1,\lambda} R_{L_2,\lambda} R_{L_3,\lambda},\;\;
        R_{L_1,\lambda} R_{L_2,\lambda} R_{L_1,\lambda} R_{L_3,\lambda},
        \;\;
        R_{L_1,\lambda} R_{L_2,\lambda} R_{L_3,\lambda} R_{L_2,\lambda} R_{L_1,\lambda}.
    \end{equation}
    On the other hand, none of the following products is a cycle:
    \begin{equation}
        R_{L_1,\lambda},\;\;
        R_{L_1,\lambda} R_{L_2,\lambda},\;\;
        R_{L_1,\lambda} R_{L_2,\lambda} R_{L_3,\lambda} R_{L_1,\lambda} R_{L_2,\lambda} R_{L_3,\lambda}.
    \end{equation}
\end{example}

It will also be convenient to set
\begin{empheq}[box=\mybluebox]{equation}
    \mathbf{L}:=\bigcap_{L\in\mathcal{L}}L
    \quad\text{and}\quad \ell:=|\mathcal{L}|.
\end{empheq}

\begin{corollary}\label{c:contraction}
    Suppose that $\mathcal{L}$ is innately regular. Then there exists a $\kappa_*>1$ such that
    \begin{equation}
        (\forall Q\in\mathcal{Q})\quad \norm{QP_{\mathbf{L}^\perp}}\leq\sqrt{1-\lambda(2-\lambda)\kappa_*^{-2(\ell-1)}} < 1.
        \label{e:251208a}
    \end{equation}
\end{corollary}
\begin{proof}
    Let $Q\in\mathcal{Q}$. By the definition of $\mathcal{Q}$, we have $Q=R_q\cdots R_0$ for some $q\in\mathbb{N}$, with $R_i\in\mathcal{R}_{\mathcal{L}}$ for all $i\in\{0,\dots,q\}$. Moreover, we also have $\mathbf{L}_q=\mathbf{L}$ and $N_q=\ell$.
    The result then follows from the ``Consequently'' part of \cref{p:sine3}.
\end{proof}

\section{{Meshulam's result} in infinite-dimensional spaces}\label{s:fixedlambda}

{Recall \cref{e:mcA} and \cref{e:mcL}.}
In this section, we assume the following: For each $A \in \mathcal{A}$, write $A
    = a + L$, where $L:=A-A$ is the closed linear subspace parallel to $A$ and
$\{a\}:= P_{L^\perp}(A) \subseteq L^\perp \cap A$;
the collection of all such translation vectors is denoted
by $\mathcal{T}$.

\begin{remark}
    Note that for each $A\in\mathcal{A}$, the set $P_{L^\perp}(A)$ is indeed a singleton. To prove this, take any $a_1,a_2\in A$. By the definition of $L$, we have that $a_1-a_2\in A-A=L$. Hence, we obtain $P_{L^\perp}(a_1-a_2)=0$, i.e., $P_{L^\perp}(a_1)=P_{L^\perp}(a_2)$.
\end{remark}

\begin{lemma}\label{l:meshulaminf}
    Let $(A_n)_{n\in\mathbb{N}}$ be a sequence drawn from $\mathcal{A}$, with associated linear subspaces $(L_n)_{n\in\mathbb{N}}$ in $\mathcal{L}$ and translation vectors $(a_n)_{n\in\mathbb{N}}$ in $\mathcal{T}$. Let $\lambda\in\left]0,2\right[$, $x_0 \in X$, and consider the sequence $(x_n)_\nnn$
    generated by
    \begin{equation}
        (\forall\nnn)\quad
        x_{n+1} := R_{A_{n},\lambda}x_{n}
        = R_nx_{n} + \lambda a_n,
    \end{equation}
    where $R_n := R_{L_n,\lambda}$. Then
    \begin{equation}
        \label{e:250428b}
        (\forall\nnn)\quad
        x_{n+1} = R_n\cdots R_0x_0 + \lambda\sum_{j=0}^n R_n\cdots R_{j+1}a_j.
    \end{equation}

\end{lemma}

{
    \let\endproof\relax
    \begin{proof}
        We will prove it by induction on $n\in\mathbb{N}$.
        For $n=0$, we have
        \begin{equation}
            x_1 =
            R_0x_0 + \lambda a_0,
        \end{equation}
        where we used the empty product convention.
        Now assume \eqref{e:250428b} holds for some $n\in\mathbb{N}$.
        Then
        \begin{subequations}
            \begin{align}
                x_{n+2}
                 & =
                R_{A_{n+1},\lambda}x_{n+1}
                =
                R_{n+1}x_{n+1} +\lambda a_{n+1}
                \\
                 & =
                R_{n+1}\Big(R_{n}\cdots R_0x_0 + \lambda\sum_{j=0}^n R_n\cdots R_{j+1}a_j\Big)
                +\lambda a_{n+1}
                \\
                 & =
                R_{n+1}R_n\cdots R_0x_0 + \lambda\sum_{j=0}^{n+1} R_{n+1}R_n\cdots R_{j+1}a_j,
            \end{align}
        \end{subequations}
        which completes the proof.
    \end{proof}
}

Now, the remaining work lies in analyzing
\begin{equation}
    \sum_{j=0}^n R_n\cdots R_{j+1}a_j.
\end{equation}
{ We will do this in the following:}

\begin{theorem}[main result for a fixed relaxation parameter]
    \label{t:meshulaminf}
    Recall \cref{e:mcA} and \cref{e:mcL}.
    Suppose that $\mathcal{L}$ is innately regular
    and that $\lambda\in\left]0,2\right[$.
    Then there exists a positive constant $C_{\mathcal{A},\lambda}<+\infty$ such that for any
    sequence $(A_n)_{n\in\mathbb{N}}$ drawn from $\mathcal{A}$ with associated linear subspaces $(L_n)_{n\in\mathbb{N}}$ in $\mathcal{L}$ and translation vectors $(a_n)_{n\in\mathbb{N}}$ in $\mathcal{T}$,
    and any starting point $x_0 \in X$, the
    sequence $(x_n)_\nnn$ generated by the iteration
    \begin{equation}
        x_{n+1} := R_{A_{n},\lambda}x_{n}
        = R_nx_{n} + \lambda a_n,
    \end{equation}
    where $R_n := R_{L_n,\lambda}$,
    satisfies
    \begin{equation}
        \label{260127a}
        (\forall n\in\mathbb{N})\quad \Big\|
        \sum_{j=0}^n R_n\cdots R_{j+1}a_j\Big\| \leq C_{\mathcal{A},\lambda};
        \quad\text{consequently,}\quad
        \|x_n\| \leq \|x_0\|+\lambda C_{\mathcal{A},\lambda}.
    \end{equation}
\end{theorem}

{
\let\endproof\relax
\begin{proof}
    {In view of \cref{e:250428b}, the left inequality in
        \cref{260127a} implies the right inequality
        in \cref{260127a}.}

    We will prove the left inequality
    in \cref{260127a} by {strong} induction on the number of subspaces.
    For the base case $\ell:=|\mathcal{L}|=1$, i.e., $\mathcal{L}=\{L\}$ and $\mathcal{A}=\{a+L\}$,
    since $a\in L^\perp$, we have
    \begin{equation}
        (\forall n\in\mathbb{N})\quad\Big\|\sum_{j=0}^n R_n\cdots R_{j+1}a_j\Big\|
        =\Big\|\sum_{j=0}^n R_{L,\lambda}^{n-j}a\Big\|=\Big\| \sum_{j=0}^n (1-\lambda)^{n-j}a\Big\|.
    \end{equation}
    This combined with the triangle inequality yields
    \begin{equation}
        (\forall n\in\mathbb{N})\quad\Big\|\sum_{j=0}^n R_n\cdots R_{j+1}a_j\Big\|\leq\sum_{j=0}^n |1-\lambda|^{n-j}\| a\|
        = \frac{1-|1-\lambda|^{n+1}}{1-|1-\lambda|}\|a\| \leq\frac{\norm{a}}{1-|1-\lambda|}<+\infty.
    \end{equation}
    Thus, the conclusion holds
    with $C_{\mathcal{A},\lambda} = \|a\|/(1-|1-\lambda|)$.

    Let $\ell\in\mathbb{N}$, $\ell\geq2$. Assume that the statement holds for all collections of closed linear subspaces $\widetilde{\mathcal{L}}$ with $|\widetilde{\mathcal{L}}|\leq \ell-1$. Now, let $\mathcal{L}$ be a collection with $|\mathcal{L}|=\ell$. Since $\mathcal{L}$ is finite, it only contains a finite number of proper subcollections. By the induction hypothesis, each proper subcollection is then associated with a positive constant. We denote $D$ to be the maximum of all such constants.

    Fix $n\in\mathbb{N}$. If the product $R_n\cdots R_1$ does not contain any cycle, then the collection of subspaces $\mathcal{L}_n$ associated with $R_n\cdots R_1$,
    { i.e., $\{L_1,\ldots,L_n\}$}, has less than $\ell$ elements.  Hence,
    \begin{equation}
        \Big\|\sum_{j=0}^n R_n\cdots R_{j+1}a_j\Big\|\leq\norm{R_n\cdots R_{1}a_0}+\Big\|\sum_{j=1}^n R_n\cdots R_{j+1}a_j\Big\|\leq \tau +D<+\infty,
    \end{equation}
    where {$\tau = \max \|\mathcal{T}\|$. }

    Now suppose that the product $R_n\cdots R_1$ contains at least one cycle.
    We scan the composition $R_n\cdots R_1$ from left to right,
    picking up the cycles as we go. Either the composition fully factors
    into cycles or there is a noncycle left:
    That is, the index list
    $(n,\ldots,1)$ is broken up into sublists
    as follows:
    \begin{equation}
        (p_{k_{n}},\dots,p_{k_n-1}+1) \cup (p_{k_{n}-1},\dots,p_{k_n-2}+1) \cup \cdots \cup (p_1,\dots,p_0+1)\cup(p_0,\dots,1),
    \end{equation}
    where $p_{k_n}=n$.
    So we have $k_n$ cycles in the composition (represented by the left $k_n$ sublists) and either $p_0=0$, which means complete
    factorization into cycles and $(0,\ldots,1)$ does not appear, or
    $p_0\geq 1$ and $(p_0,\ldots,1)$ represents the noncycle $R_{p_0}\cdots R_1$.

    Note that
    for each $i\in\{0,\dots,k_n\}$, $p_i$ is the largest index $j\in\{0,\dots,n\}$ such that the product $R_n\cdots R_{j+1}$ is fully factored into exactly $k_n - i$ cycles (with no remaining noncyle).

    For $0\leq r\leq s \leq n$, we define
    \begin{equation}
        q(s,r) := \sum_{j=r}^s R_s\cdots R_{j+1}a_j.
    \end{equation}
    The empty product convention gives $q(r,r)=a_r$.
    Our goal is to get $\norm{q(n,0)}$ universally bounded.

    By the definition of $(a_n)_{n\in\mathbb{N}}$, we have $a_j\in L_j^\perp\subseteq\mathbf{L}^\perp$ for all $j\in\{0,\ldots,n\}$.
    Hence, we get
    \begin{equation}
        (\forall 0\leq r\leq s\leq n)\quad q(s,r)=\sum_{j=r}^s R_s\cdots R_{j+1}a_j=\sum_{j=r}^s R_s\cdots R_{j+1}P_{\mathbf{L}^\perp}a_j.
    \end{equation}

    Observe that
    \begin{subequations}
        \label{250506a}
        \begin{align}
            q(n,0)=q(p_{k_n},0)
             & =
            \sum_{j=0}^{p_{k_n}} R_n\cdots R_{j+1}P_{\mathbf{L}^\perp}a_j                \\
             & =
            \sum_{j=p_{k_n-1}+1}^{p_{k_n}} R_n\cdots R_{j+1}P_{\mathbf{L}^\perp}a_j+\sum_{j=0}^{p_{k_n-1}} R_n\cdots R_{p_{k_n-1}+1}\cdots R_{j+1}P_{\mathbf{L}^\perp}a_j
            \\
             & =\sum_{j=p_{k_n-1}+1}^{p_{k_n}} R_n\cdots R_{j+1}P_{\mathbf{L}^\perp}a_j+
            R_n\cdots R_{p_{k_n-1}+1}P_{\mathbf{L}^\perp}\sum_{j=0}^{p_{k_n-1}} R_{p_{k_n-1}}\cdots R_{j+1}P_{\mathbf{L}^\perp}a_j,
        \end{align}
    \end{subequations}
    where we used \cref{f:commute} in the last equality. Continuing in this fashion, we arrive at
    \begin{subequations}
        \label{20251121b}
        \begin{align}
            q(n,0) & =
            q(p_{k_n},p_{k_n-1}+1)+R_n\cdots R_{p_{k_n-1}+1}P_{\mathbf{L}^\perp}q(p_{k_n-1},0)                                                                                         \\
                   & =q(p_{k_n},p_{k_n-1}+1)                                                                                                                                           \\
                   & \ \ \ \ +R_n\cdots R_{p_{k_n-1}+1}P_{\mathbf{L}^\perp}\left(q(p_{k_n-1},p_{k_n-2}+1)+R_{p_{k_n}-1}\cdots R_{p_{k_n-2}+1}P_{\mathbf{L}^\perp}q(p_{k_n-2},0)\right) \\
                   & =q(p_{k_n},p_{k_n-1}+1)+R_n\cdots R_{p_{k_n-1}+1}P_{\mathbf{L}^\perp}q(p_{k_n-1},p_{k_n-2}+1)                                                                     \\
                   & \ \ \ \ +R_n\cdots R_{p_{k_n-2}+1}P_{\mathbf{L}^\perp}q(p_{k_n-2},0)                                                                                              \\
                   & =q(p_{k_n},p_{k_n-1}+1)+R_n\cdots R_{p_{k_n-1}+1}P_{\mathbf{L}^\perp}q(p_{k_n-1},p_{k_n-2}+1)                                                                     \\
                   & \ \ \ \ +R_n\cdots R_{p_{k_n-2}+1}P_{\mathbf{L}^\perp}q(p_{k_n-2},p_{k_n-3}+1)                                                                                    \\
                   & \ \ \ \ +R_n\cdots R_{p_{k_n-3}+1}P_{\mathbf{L}^\perp}q(p_{k_n-3},0)                                                                                              \\
                   & \ \ \vdots                                                                                                                                                        \\
                   & =\sum_{i=1}^{k_n} R_n\cdots R_{p_{i}+1}P_{\mathbf{L}^\perp} q(p_{i},p_{i-1}+1) + R_n\dots R_{p_0+1}P_{\mathbf{L}^\perp}q(p_0,0).
        \end{align}
    \end{subequations}
    For all $i\in\{1,\ldots,k_n\}$,
    by the definition of $p_i$, we have that $R_{p_{i}}\cdots R_{p_{i-1}+2}$ does not contain any cycle. Hence, the collection of subspaces associated with $R_{p_{i}}\cdots R_{p_{i-1}+2}$ has less than $\ell$ elements. The induction hypothesis then implies that
    \begin{subequations}
        \label{20251121a}
        \begin{align}
            (\forall i\in\{1,\dots,k_n\})\quad \norm{q(p_{i},p_{i-1}+1)} & =\Big\|\sum_{j=p_{i-1}+1}^{p_i} R_{p_i}\cdots R_{j+1}a_j\Big\|                                                      \\
                                                                         & \leq\norm{R_{p_i}\cdots R_{p_{i-1}+2}a_{p_{i-1}+1}}+\Big\|{\sum_{j=p_{i-1}+2}^{p_i} R_{p_i}\cdots R_{j+1}a_j}\Big\| \\
                                                                         & \leq\tau +D.
        \end{align}
    \end{subequations}

    Since $R_{p_0} \cdots R_1$ corresponds to the remainder in the cycle decomposition, it also contains no cycle. Hence, by an argument similar to \eqref{20251121a}, we obtain
    \begin{equation}
        \norm{q(p_0,0)}\leq\tau+D.\label{20251207a}
    \end{equation}

    Recall that for all $i\in\{0,\dots,k_n\}$, the composition $R_n\cdots R_{p_i+1}$ factors into exactly $k_n - i$ cycles. We now pick up $\kappa_*>1$
    from \cref{c:contraction} for
    $\mathcal{L}$. We claim that
    \begin{equation}
        (\forall i\in\{0,\dots,k_n\})\quad\norm{R_n\dots R_{p_{i}+1}P_{\mathbf{L}^\perp}}\leq\left(1-\lambda(2-\lambda)\kappa_*^{-2(\ell-1)}\right)^{(k_n-i)/2}.\label{20251121c}
    \end{equation}
    Indeed, \cref{20251121c} is true for $i=k_n$ because
    $\|P_{\mathbf{L}^\perp}\|\leq 1$ (see \cref{f:reflection1}).

    For $i\in\{0,\ldots,k_n-1\}$, we have
    \begin{equation}
        R_n\cdots R_{p_i+1}=Q_1\cdots Q_{k_n-i},
    \end{equation}
    where each $Q_j$ is a cycle, for $j\in\{1,\dots, k_n-i\}$. Since $P_{\mathbf{L}^\perp}$ commutes with $P_L$ for all $L\in\mathcal{L}$, it follows from \cref{c:contraction} that
    \begin{subequations}
        \begin{align}
            \|R_n\cdots R_{p_i+1}P_{\mathbf{L}^\perp}\| & =\|Q_1\cdots Q_{k_n-i}P_{\mathbf{L}^\perp}\|                            \\
                                                        & =\|Q_1P_{\mathbf{L}^\perp}\cdots Q_{k_n-i}P_{\mathbf{L}^\perp}\|        \\
                                                        & \leq\|Q_1P_{\mathbf{L}^\perp}\|\cdots\|Q_{k_n-i}P_{\mathbf{L}^\perp}\|  \\
                                                        & \leq\left(1-\lambda(2-\lambda)\kappa_*^{-2(\ell-1)}\right)^{(k_n-i)/2}.
        \end{align}
    \end{subequations}
    Next, we estimate
    \begin{subequations}
        \begin{align}
            \norm{q(n,0)}
             & =\Big\|R_n\cdots R_{p_0+1}P_{\mathbf{L}^\perp}q(p_0,0)+\sum_{i={1}}^{k_n} R_n\cdots R_{p_{i}+1}P_{\mathbf{L}^\perp} q(p_{i},p_{i-1}+1)\Big\| \tag{by \cref{20251121b}} \\
             & \leq\norm{R_n\dots R_{p_0+1}P_{\mathbf{L}^\perp}q(p_0,0)}+\sum_{i={1}}^{k_n} \norm{R_n\cdots R_{p_{i}+1}P_{\mathbf{L}^\perp} q(p_{i},p_{i-1}+1)}
            \tag{triangle inequality}                                                                                                                                                 \\
             & \leq\norm{R_n\dots R_{p_0+1}P_{\mathbf{L}^\perp}}\norm{q(p_0,0)}+\sum_{i={1}}^{k_n} \norm{R_n\cdots R_{p_{i}+1}P_{\mathbf{L}^\perp}}\norm{q(p_{i},p_{i-1}+1)}
            \notag                                                                                                                                                                    \\
             & \leq\left(1-\lambda(2-\lambda)\kappa_*^{-2(\ell-1)}\right)^{k_n/2}(\tau+D)
            \notag                                                                                                                                                                    \\
             & \ \ \ \ \ +\sum_{i={1}}^{k_n}\left(1-\lambda(2-\lambda)\kappa_*^{-2(\ell-1)}\right)^{(k_n-i)/2}(\tau+D)
            \tag{by \cref{20251121c}, \cref{20251121a}, and \cref{20251207a}}                                                                                                         \\
             & =\sum_{i=0}^{k_n}\left(1-\lambda(2-\lambda)\kappa_*^{-2(\ell-1)}\right)^{(k_n-i)/2}(\tau+D)
            \notag                                                                                                                                                                    \\
             & \leq \frac{\tau+D}{1-\sqrt{1-\lambda(2-\lambda)\kappa_*^{-2(\ell-1)}}}.
            \tag{by \cref{e:251208a} and Geometric Series}
        \end{align}
    \end{subequations}
    This, \cref{l:meshulaminf}, and \cref{f:reflection1} yield the conclusion with
    \begin{equation}
        C_{\mathcal{A},\lambda} := \frac{\tau+D}{1-\sqrt{1-\lambda(2-\lambda)\kappa_*^{-2(\ell-1)}}}.\tag*{$\blacksquare$}
    \end{equation}
\end{proof}

Using a convexity argument, we now readily obtain
the following generalization of \cref{t:meshulaminf}
concerning the boundedness of the sequence generated
by relaxed projections:

\begin{corollary}[main result for varying relaxation parameters]
    \label{t:meshulaminfvar}
    Recall \cref{e:mcA} and \cref{e:mcL}.
    Suppose that $\mathcal{L}$ is innately regular and
    that $\lambda\in\left] 0,2\right[$.
    Then there exists a positive
    constant $C_{\mathcal{A},\lambda}<+\infty$ such that
    for
    any sequence $(A_n)_{n\in\mathbb{N}}$ drawn from $\mathcal{A}$,
    any sequence $(\lambda_n)_{n\in\mathbb{N}}$ in $[0,\lambda]$, and any
    starting point $x_0\in X$, the
    sequence generated by the iteration
    \begin{equation}
        x_{n+1} := R_{A_{n},\lambda_n}x_{n},
    \end{equation}
    satisfies
    \begin{equation}
        (\forall\nnn)\quad
        \|x_n\| \leq \|x_0\|+\lambda C_{\mathcal{A},\lambda}.
    \end{equation}
\end{corollary}
\begin{proof}
    Note that
    \begin{subequations}
        \begin{align}
            x_{n+1} & = R_{A_{n},\lambda_n}x_{n}
            = (1-\lambda_n)x_n+\lambda_nP_{A_n}x_n,                                                                  \\
                    & = (1-\lambda_n/\lambda)x_n+(\lambda_n/\lambda-\lambda_n)x_n+\lambda_nP_{A_n}x_n                \\
                    & =(1-\lambda_n/\lambda)x_n+\frac{\lambda_n}{\lambda}\big((1-\lambda)x_n+\lambda P_{A_n}x_n\big) \\
                    & =\big((1-\mu_n)\Id+\mu_nR_{A_n,\lambda}\big)x_n,
        \end{align}
    \end{subequations}
    where $\mu_n := \lambda_n/\lambda\in[0,1]$.
    This implies $x_1=(1-\mu_0)x_0+\mu_0R_{A_0,\lambda} \in\conv\{x_0,R_{A_0,\lambda}x_0\}$
    and
    \begin{subequations}
        \begin{align}
            x_2
             & =
            \big((1-\mu_1)\Id + \mu_1R_{A_1,\lambda}\big)x_1                                    \\
             & =
            \big((1-\mu_1)\Id + \mu_1R_{A_1,\lambda}\big)
            \big((1-\mu_0)x_0+\mu_0 R_{A_0,\lambda}x_0\big)                                     \\
             & =
            (1-\mu_1)\big((1-\mu_0)x_0 + \mu_0R_{A_0,\lambda}x_0\big)
            +
            \mu_1R_{A_1,\lambda}\big((1-\mu_0)x_0 + \mu_0R_{A_0,\lambda}x_0\big)                \\
             & =
            (1-\mu_1)\big((1-\mu_0)x_0 + \mu_0R_{A_0,\lambda}x_0\big)
            +
            \mu_1\big((1-\mu_0)R_{A_1,\lambda}x_0 + \mu_0R_{A_1,\lambda}R_{A_0,\lambda}x_0\big) \\
             & =
            (1-\mu_1)(1-\mu_0)x_0
            + (1-\mu_1)\mu_0 R_{A_0,\lambda}x_0
            + \mu_1(1-\mu_0)R_{A_1,\lambda}x_0
            + \mu_1\mu_0R_{A_1,\lambda}R_{A_0,\lambda}x_0,
        \end{align}
    \end{subequations}
    which is in the convex hull of $\{x_0,R_{A_0,\lambda}x_0,R_{A_1,\lambda}x_0,R_{A_1,\lambda}R_{A_0,\lambda}x_0\}$.
    Induction on $n$ yields in general
    \begin{equation}
        x_n = \sum_{J\subseteq \{1,\ldots,n\}}
        \Big(\prod_{k\notin J} (1-\mu_k)\Big)
        \Big(\prod_{j\in J} \mu_j\Big)
        \Big(\prod_{j\in J} R_{A_j,\lambda}\Big)
        x_0
    \end{equation}
    where if $J=\{j_1,\ldots,j_k\}$ and
    $j_1<\cdots<j_k$, then
    $R_{A_J,\lambda} := \prod_{j\in J} R_{A_j,\lambda}:= R_{A_{j_k},\lambda}\cdots R_{A_{j_1},\lambda}$.
    Hence $x_n$ lies in the convex hull of
    $\{R_{A_J,\lambda}x_0\}_{J\subseteq\{1,\ldots,n\}}$.
    Since $\mathcal{L}$ is innately regular, by \cref{t:meshulaminf}, $\{R_{A_J,\lambda}x\}_{J\subseteq\{1,\ldots,n\}}$
    lies in the (convex!) ball of radius $\|x_0\|+\lambda C_{\mathcal{A},\lambda}$ centered at $0$ for all $n\in\mathbb{N}$. Consequently, $(x_n)_{n\in\mathbb{N}}$ also lies in that ball and we are done.
\end{proof}

\section{Applications and limiting examples}
\label{s:ex}

\subsection*{Connection to randomized block Kaczmarz methods}

Consider the problem of solving a linear system
\begin{equation}
    \label{e:RKP}
    Mx = b,
\end{equation}
where $M\in\RR^{p\times q}$ and $b\in\RR^p$.
\emph{Randomized block Kaczmarz algorithms}
tackle \cref{e:RKP} by producing a sequence whose
terms are updated by projecting onto the randomly chosen affine subspaces of the form
$M_Ix=b_I$, where $I$ is a block of indices drawn from
$\{1,\ldots,p\}$, and $M_I$ (resp.\ $b_I$) is the matrix
(resp.\ vector) created from $M$ (resp.\ $b$) by retaining
only entries corresponding to the row indices~$I$.
(The original randomized Kaczmarz algorithm arises
if each block of indices is a singleton, i.e., the
affine subspaces are hyperplanes.)
Randomized block Kaczmarz methods are
now well understood even in the inconsistent case (when \cref{e:RKP} has no solution). Typical convergence results assert that
\begin{equation}
    \text{
        the sequence $(x_n)_\nnn$ generated by
        randomized block Kaczmarz is \emph{bounded in expectation},
    }
\end{equation}
along with estimates to least-squares solutions; see, e.g.,
the paper by Needell and Tropp \cite{NT},
and references therein.
We note that \cref{f:Meshulam} strengthens this not only
to almost sure boundedness but even to
\begin{equation}
    \text{
        the sequence $(x_n)_\nnn$ generated by
        randomized block Kaczmarz is \emph{always bounded},
    }
\end{equation}
which is an observation we have not seen explicitly stated
in the literature on randomized block Kaczmarz algorithms.

\subsection*{A cyclic result}

In the setting of \cref{t:meshulaminf}, if we do not randomly
pick relaxed projectors but rather iterate cyclically, then
the resulting sequence converges \emph{linearly} as we now show:

\begin{theorem}[innate regularity and linear convergence of
        cyclic relaxed projections]\label{t:cyclic}
    Recall \cref{e:mcA} and \cref{e:mcL}.
    Suppose that $\mathcal{L}$ is innately regular and that
    $\lambda\in\left]0,2\right[$.
    Let $Q$ be a finite composition of relaxed projectors
    drawn from $\mathcal{R}_{\mathcal{A},\lambda}$.
    Then $\Fix Q\neq\varnothing$ and
    for every $x_0\in X$, the sequence
    $(Q^nx_0)_\nnn$ converges linearly to
    $P_{\Fix Q}(x_0)$.
\end{theorem}
\begin{proof}
    By \cref{t:meshulaminf}, the sequence $(Q^nx_0)_\nnn$ is bounded.
    By \cite[Theorem~1]{BrowderP}, $\Fix Q\neq\varnothing$.
    Let $y_0\in \Fix Q$, and
    let $T$ be the associated composition of $Q$,
    where the affine subspaces are replaced by the corresponding
    parallel linear spaces.
    By \cite[Corollary~3.3.(iii)]{BEdwards},
    $(\forall\nnn)$ $Q^nx_0 = T^n(x_0-y_0)+y_0$.
    The innate regularity of $\mathcal{L}$ coupled with
    \cite[Theorem~5.7]{BB1996} and \cite[Proposition~5.9(ii)]{Bauschke2017}
    yield pointwise linear convergence
    of the iterates of $T$ to $P_{\Fix T}$.
    Finally, \cite[Theorem~3.3]{Lukens} yields
    pointwise linear convergence of the iterates of $Q$ to
    $P_{\Fix Q}$.
\end{proof}

\subsection*{An inconsistent linear system in Euclidean space}
We illustrate \cref{t:cyclic} by plotting the behavior of relaxed projections onto a randomly generated family of affine hyperplanes with empty intersection. We generate $M\in\mathbb{R}^{15\times 10}$ with i.i.d.\ standard normal entries and normalize each row, and $b\in\mathbb{R}^{15}$ with i.i.d.\ standard normal entries. The affine hyperplanes are $A_i := \{x\in\mathbb{R}^{10} \mid \langle a_i,x\rangle = b_i\}$, where $a_i$ is the $i$th row of $M$. Starting from $x_0=0$, we construct the sequence of iterates $x_{n+1} = R_{A_{i_n},\lambda} x_n$, where $i_n$ is chosen uniformly for the randomized method or cyclically for the cyclic method. In the cyclic plot, we also highlight the subsequence $(Q^n x_0)_\nnn$, where $Q = R_{A_{15},\lambda}\cdots R_{A_1,\lambda}$. We use relaxation parameters $\lambda \in \{0.5,1,1.5\}$ and run $3000$ iterations (i.e., $200$ applications of $Q$). For visualization, only the first two coordinates of the iterates are plotted
in \cref{fig:traj}. The random orbit appears to be more dense and exploratory
than its cyclic counterpart.

\begin{figure}[H]
    \centering
    \begin{subfigure}{0.45\linewidth}
        \includegraphics[width=\linewidth]{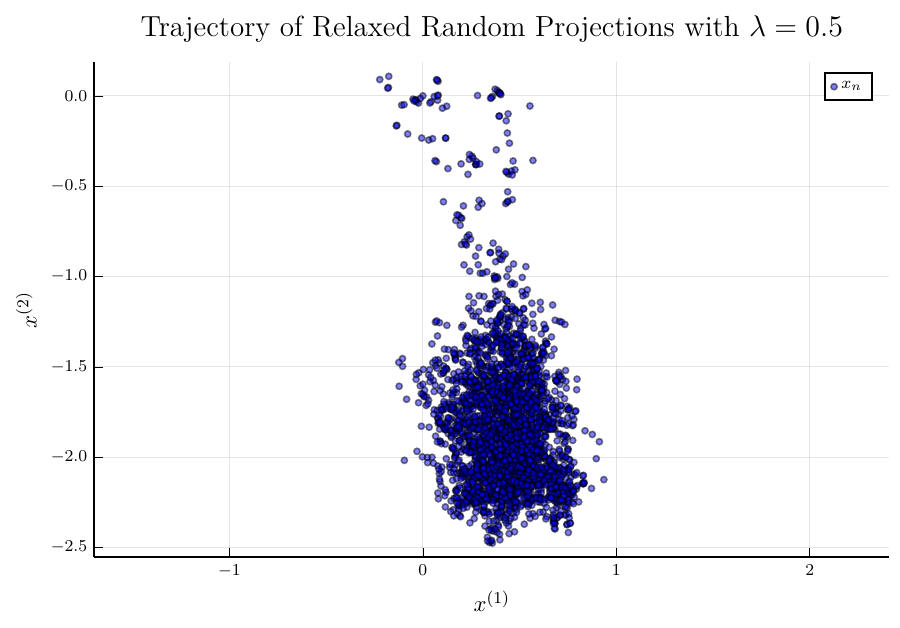}
    \end{subfigure}
    \hfill
    \begin{subfigure}{0.45\linewidth}
        \includegraphics[width=\linewidth]{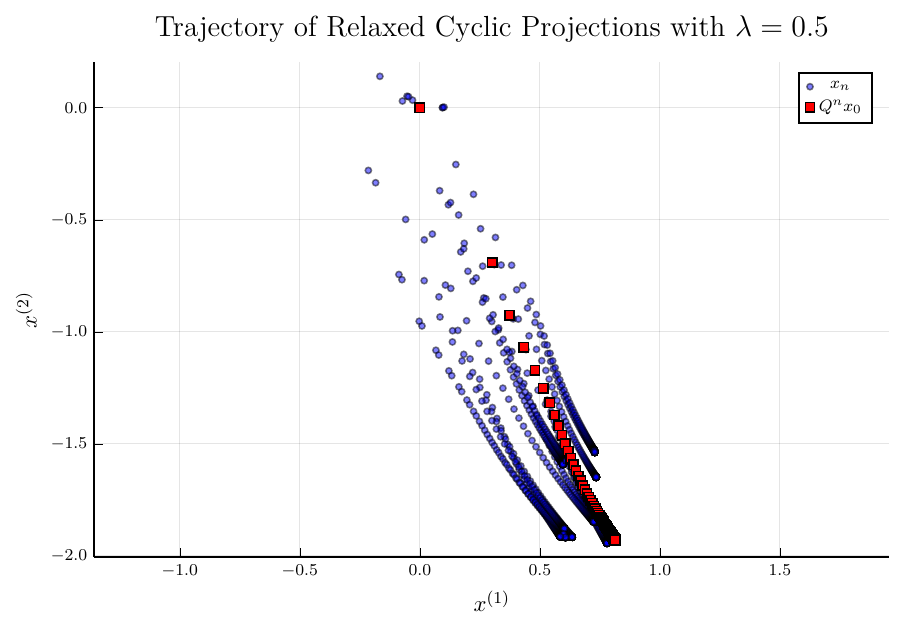}
    \end{subfigure}

    \vspace{0.05em} 

    \begin{subfigure}{0.45\linewidth}
        \includegraphics[width=\linewidth]{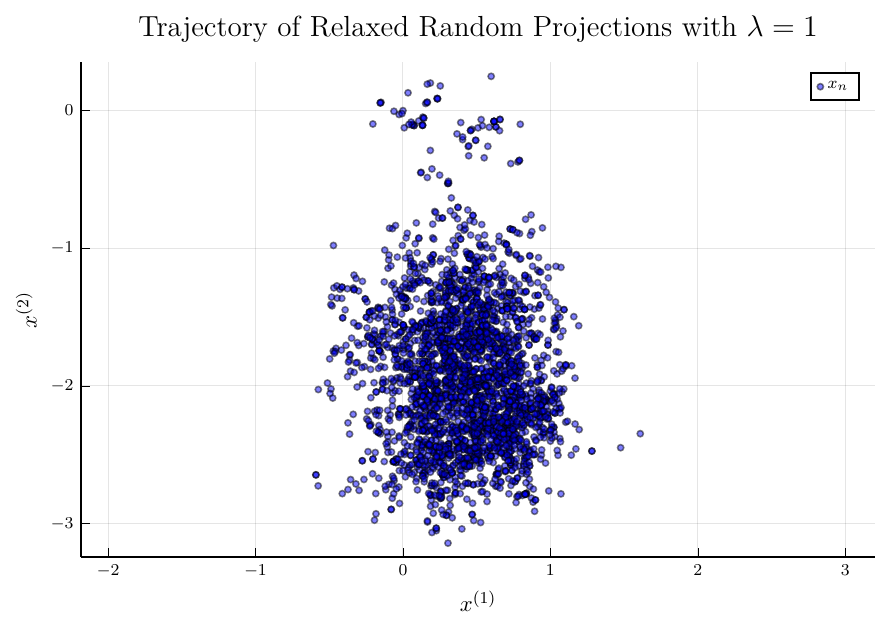}
    \end{subfigure}
    \hfill
    \begin{subfigure}{0.45\linewidth}
        \includegraphics[width=\linewidth]{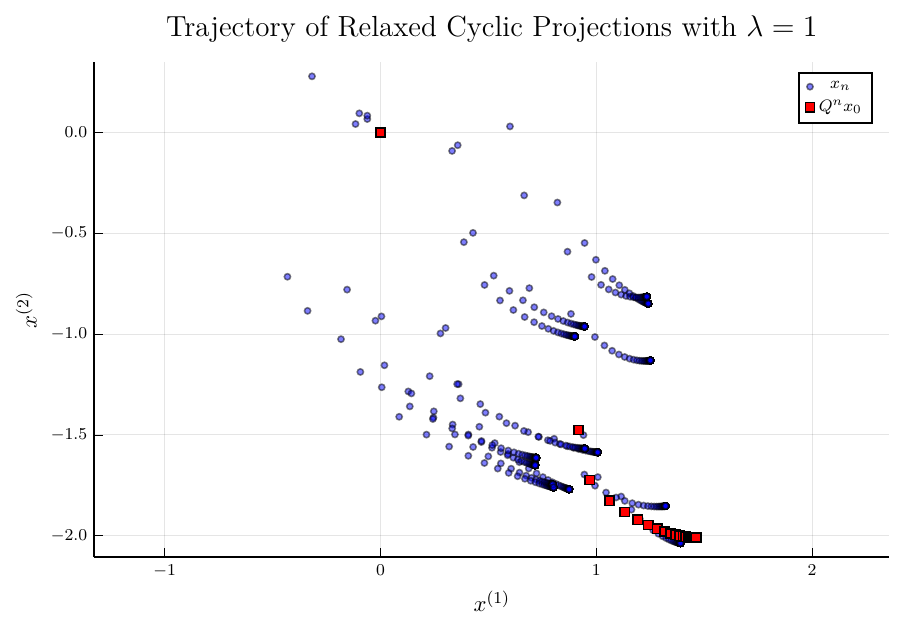}
    \end{subfigure}

    \vspace{0.05em} 

    \begin{subfigure}{0.45\linewidth}
        \includegraphics[width=\linewidth]{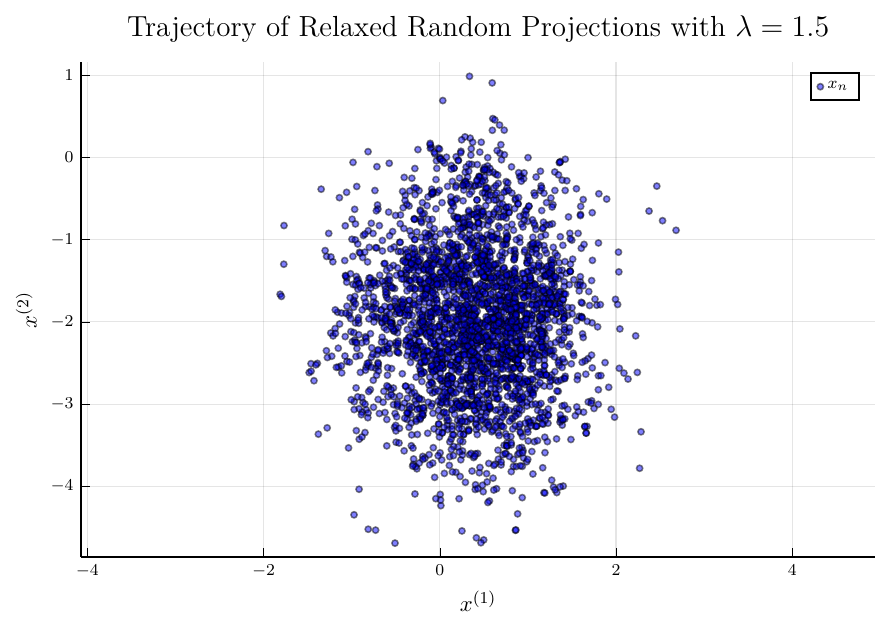}
    \end{subfigure}
    \hfill
    \begin{subfigure}{0.45\linewidth}
        \includegraphics[width=\linewidth]{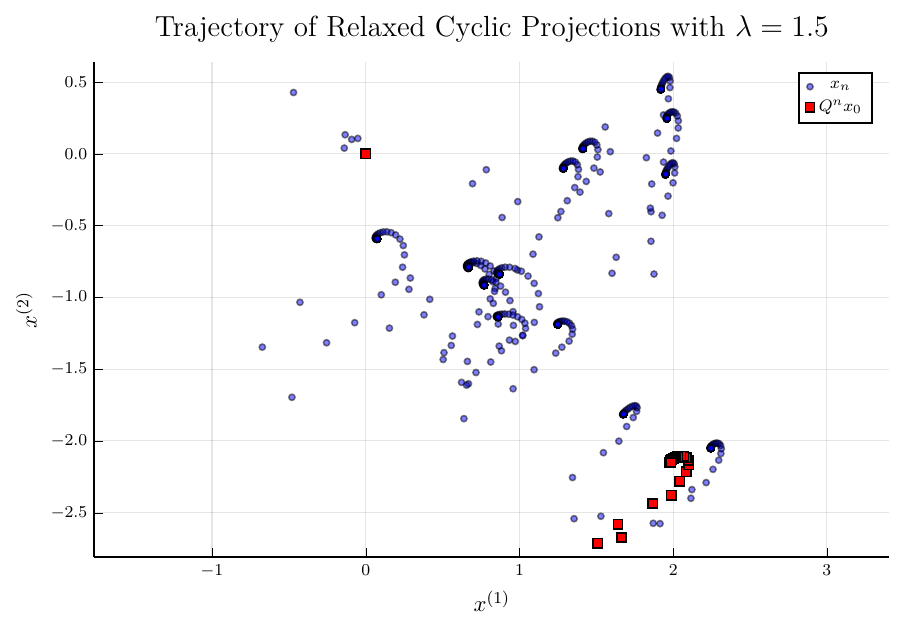}
    \end{subfigure}

    \caption{The first two coordinates of the relaxed random and cyclic projection sequences.}
    \label{fig:traj}
\end{figure}

\subsection*{An inconsistent linear inverse problem in $L_2[0,1]$}

Now suppose that $X=L_2[0,1]$, which is an infinite-dimensional separable real Hilbert space, and that we are given
four hyperplanes
$A_i = \menge{x \in X}{\scal{a_i}{x}=b_i}$,
with $(b_1,b_2,b_3,b_4)=(1,1,1,1)$ and
\begin{equation}
    a_1(t):= \sqrt{2}\sin(2\pi t),
    \;\;
    a_2(t):= \sqrt{2}\cos(2\pi t),
    \;\;
    a_3(t):= \sqrt{2}\cos(6\pi t),
    \;\;
    a_4 := a_1+a_2+a_3.
\end{equation}
(The functions $a_1,a_2,a_3$ form an orthonormal system and
so $\|a_4\|=\sqrt{3}$.)
Because $b_1+b_2+b_3\neq b_4$, we deduce that
$A_1\cap A_2\cap A_3\cap A_4 =\varnothing$, i.e.,
the corresponding linear inverse problem is inconsistent\footnote{For more on (possibly inconsistent) linear inverse problems, see \cite{BDMP}.}.
Starting from $x_0=0\in X$ and using the relaxation parameter
$\lambda=1.5$, we visualize
some of the random and the cyclic iterates in \cref{fig:3} and \cref{fig:2}, respectively. Similarly to the behaviour
in the finite-dimensional experiment, the cyclic iterates
converge to a cycle while the random iterates once again appear to be more exploratory.

\begin{figure}[H]
    \includegraphics[width=\linewidth]{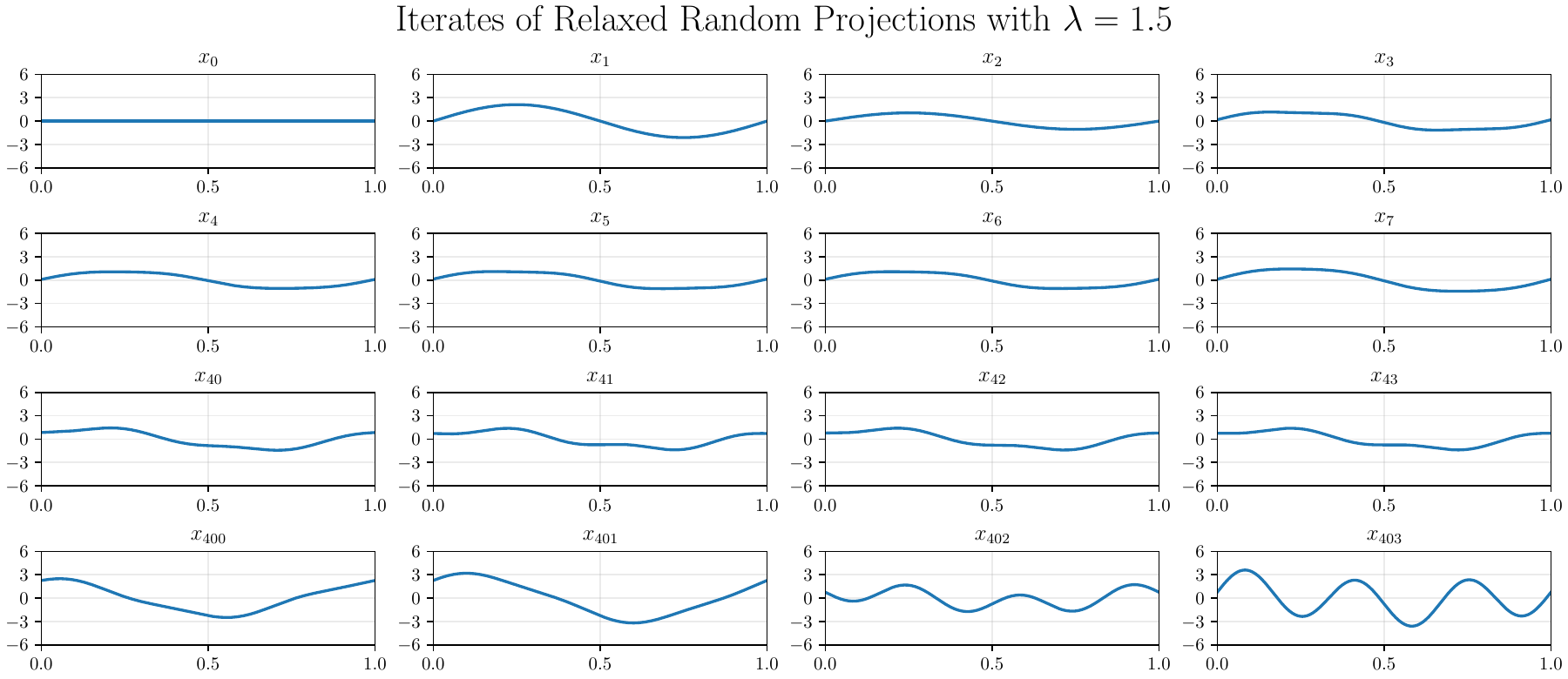}
    \caption{Selected iterates from the relaxed random projection sequence}
    \label{fig:3}
\end{figure}

\begin{figure}[H]
    \includegraphics[width=\linewidth]{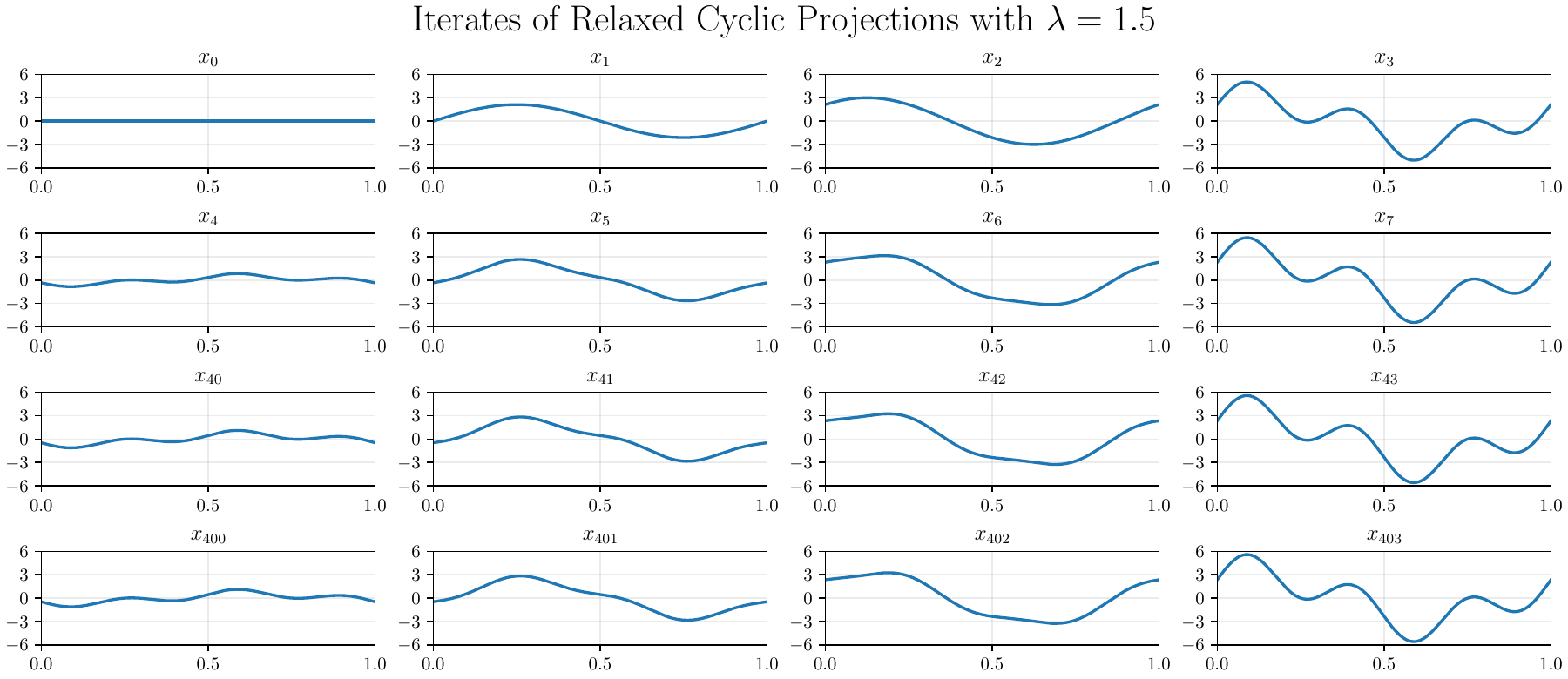}
    \caption{Selected iterates from the relaxed cyclic projection sequence}
    \label{fig:2}
\end{figure}

\subsection*{Concluding comments}

We conclude this paper by pointing out a variant
of \cref{f:Meshulam} as well as a limiting example.

\begin{remark}[polyhedral sets]
    \label{r:ps}
    Consider \cref{f:Meshulam}.
    \begin{enumerate}
        \item
              \label{r:ps1}
              One can show (see \cite[Theorem~3.2]{BTBanff}]) that
              \cref{f:Meshulam} remains true if
              $\mathcal{A}$ is replaced by a nonempty finite collection
              of \emph{polyhedral subsets} of $X$.
        \item
              The result mentioned in \cref{r:ps1} is a variant of
              \cref{t:meshulaminf}; however, neither implies the other.
    \end{enumerate}
\end{remark}

The astute reader will wonder whether the innate regularity
assumption is needed. The following limiting example shows
that \emph{some} additional assumption is required to
guarantee boundedness of the sequence generated in \cref{t:meshulaminf}:

\begin{example}[\cref{t:meshulaminf} may fail without innate regularity]
    {\rm \cite[Example~4.2]{BTBanff}}\label{e:noregular}
    Following \cite[Example~4.3]{02}, there exists an instance of the Hilbert space $X$ that contains two closed affine subspaces
    $A_1$ and $A_2$ such that their corresponding linear subspaces
    $L_1,L_2$ form a collection $\mathcal{L}=\{L_1,L_2\}$
    that is not innately regular. The ``gap''
    $\inf\|A_1-A_2\|$ between $A_1,A_2$ is equal to $1$ but
    the infimum is not attained. Now let  $x_0\in X$ and generate the sequence of \emph{alternating projections} via
    \begin{equation}
        x_{2n+1} := P_{A_1}x_{2n}
        \;\;\text{and}\;\;
        x_{2n+2} := P_{A_2}x_{2n+1}.
    \end{equation}
    By \cite[Corollary~4.6]{02}, we have $\|x_n\|\to\infty$.
\end{example}

Finally, we conclude with a comment on the sequence of
relaxation parameters:

\begin{remark}[relaxation parameters]
    In \cref{t:meshulaminfvar}, we assumed that
    the sequence $(\lambda_n)_\nnn$ of relaxation parameters
    satisfies $\sup_{\nnn}\lambda_n < 2$.
    We point out that \cite[Section~5]{BTBanff} identifies several scenarios in which the sequence $(\lambda_n)_\nnn$ in \cref{t:meshulaminfvar} satisfies $\varlimsup_{n}\lambda_n = 2$, and the corresponding iterates $(x_n)_{\nnn}$ exhibit different behaviors: they may be constant, convergent, bounded but not convergent, or unbounded.
\end{remark}

\section*{Acknowledgments}
\small
The authors thank Dr.~Daniel Reem for his helpful
comments on earlier versions of this manuscript,
Dr.~Francisco Arag\'on Artacho
for referring us to \cite[Section~5]{AACT} which inspired
us to consider the example in $L_2[0,1]$ in \cref{s:ex},
Dr.~Simeon Reich for referring us to \cite{RZBams}
and \cite{RZJAA} which enhanced \cref{r:2},
and two anonymous referees for their helpful
comments.
The plots were created with the help of
\texttt{ChatGPT}.
The research of HHB was partially supported by a Discovery Grant
of the Natural Sciences and Engineering Research Council of
Canada.

\end{document}